\documentclass{amsart}

\usepackage{amssymb}

\newcommand\Curr{\mathop {\fam 0 Cur}\nolimits}
\newcommand\Virr{\mathop {\fam 0 Vir}\nolimits}
\newcommand\Cend{\mathop {\fam 0 Cend}\nolimits}
\newcommand\Coeff{\mathop {\fam 0 Coeff}\nolimits}
\newcommand\gc{\mathop {\fam 0 gc}\nolimits}
\newcommand\End{\mathop {\fam 0 End}\nolimits}
\newcommand\codim{\mathop {\fam 0 codim}\nolimits}

\newcommand\wt{\mathop {\fam 0 wt}\nolimits}
\newcommand\Ress{\mathop {\fam 0 Res}\nolimits}

\newcommand\Idd{\mathop {\fam 0 Id}\nolimits}
\newcommand\rank{\mathop {\fam 0 rank}\nolimits}
\newcommand\charact{\mathop {\fam 0 char}\nolimits}
\newcommand\diag{\mathop {\fam 0 diag}\nolimits}
\newcommand\sll{\mathop {\fam 0 sl}\nolimits}
\newcommand\oo[1]{\mathrel {\circ_{#1}}}
\newcommand\so[1]{\mathrel {{\scriptscriptstyle\square}_{#1}}}
\renewcommand\mod{\mathrel {\fam 0 mod}}

\let\citep=\cite

\newcommand\Cset{\mathbb C}
\newcommand\Zset{\mathbb Z}

\newtheorem{thm}{Theorem}
\newtheorem{lem}[thm]{Lemma}
\newtheorem{prop}[thm]{Proposition}
\newtheorem{cor}[thm]{Corollary}

\theoremstyle{definition}
\newtheorem{defn}[thm]{Definition}
\newtheorem{exmp}[thm]{Example}

\numberwithin{equation}{section}
\numberwithin{thm}{section}

\begin{document}

\title{Universally defined representations of Lie conformal superalgebras}

%\thanks{}

\author{Pavel Kolesnikov}
\address{Sobolev Institute of Mathematics,
Akad. Koptyug prosp., 4, Novosibirsk, 630090, Russia}

\email{pavelsk@math.nsc.ru}

\begin{abstract}
We distinguish a class of irreducible
finite representations of conformal Lie (super)algebras. These representations
(called universally defined) are the simplest ones from the computational
point of view: a universally defined representation of a conformal Lie
(super)algebra $L$ is completely determined by commutation relations of $L$
and by the requirement of associative locality of generators.
We describe such representations for conformal superalgebras $W_n$, $n\ge 0$,
with respect to a natural set of generators. We also consider the problem
for superalgebras $K_n$. In particular, we find a universally defined
representation for the Neveu--Schwartz conformal superalgebra $K_1$
and show that the analogues of this representation for $n\ge 2$
are not universally defined.
\end{abstract}



\maketitle

\section{Introduction}

Conformal Lie (super)algebras, as introduced in \citep{K1},
 provide a formal language to operate with certain
infinite-dimensional Lie (super)algebras in conformal field theory
and (super)string theory. From the algebraic point of view,
a conformal algebra is an algebraic system based on a linear space
$C$ over $\Cset $
endowed with a family of bilinear operations (conformal products)
$(\cdot \oo{n} \cdot)$, where $n$ ranges over the set
$\Zset_+$ of non-negative integers.

In conformal field theory, $C$ is assumed to be a space of
pairwise mutually local fields; if $a,b\in C$ then the sequence
$a\oo{n} b$, $n\in \Zset_+$, encodes the singular part of the operator
product expansion (OPE) of $a$ and~$b$. The properties of OPE
give rise to the axioms of (Lie) conformal algebras \citep{K1,K2}.
Roughly speaking, a Lie conformal algebra is a ``singular part''
of a vertex operator algebra~\citep{Bor,FLM}.

The problem of classification of simple and semisimple
Lie conformal superalgebras of finite type
(i.e., finitely generated modules over $\Cset [D]$)
was solved in
\citep{CK2,DK,FK,FKR}, see also \citep{K3}.

An important role in conformal field theory and representation
theory belongs to vertex operator realizations of infinite-dimensional
Lie (super)algebras. These constructions lead to the notion of a
conformal module \citep{CK,K1} which is equivalent to the notion
of a module over a conformal algebra.
Irreducible conformal modules over Lie conformal superalgebras have been
studied in a series of papers. In particular, the complete list
of irreducible modules over several simple conformal Lie superalgebras
of finite type was found in \citep{CK,CL}; extensions of such modules
were described in \citep{CKW1,CKW2}.

In this paper, we develop a combinatorial approach to representations
of Lie conformal superalgebras. In the case of ordinary algebras,
every representation of a Lie algebra $L$ gives rise to a representation
of its universal associative enveloping  algebra $U(L)$.
This is not the case for conformal algebras since there is no
universal associative enveloping conformal algebra for a Lie conformal
algebra.

However, given a conformal Lie (super)algebra $C$
generated by its subset $B$
one may consider a class of associative envelopes of $C$  with a
restriction on the locality function on $B$ \citep{Ro2}. There exists
the universal envelope in that class, so we obtain a lattice
of universal envelopes of $C$. Every irreducible conformal $C$-module of
finite type corresponds to a simple homomorphic image
of a universal envelope $U$ of $C$, so the first points
of interest are the minimal (non-trivial) elements of the lattice of
universal envelopes, namely, simple universal envelopes of
at most linear growth.  Every universal envelope of this kind
defines a representation which is called universally defined.

We describe all universally defined representations
of Lie conformal superalgebras $W_n$, $n\ge 0$, with respect
to a natural set of generators. It turns out that there exists
only one universally defined representation of
$W_0$ (the Virasoro conformal algebra) and two inequivalent
representations of $W_n$, $n>0$.
We also show that the induced representations of $K_n\subset W_n$, $n\ge 1$,
are irreducible for any $n\ne 2$,
equivalent to a universally defined representation for $n=1$
(i.e., for the Neveu--Schwartz conformal superalgebra),
but for $n\ge 2$ neither of these representations is universally defined.

\section{Main definitions}

\begin{defn}[Kac, 1997]\label{defn-conf}
A {\em conformal algebra\/} is a linear space $C$
over a field $\Bbbk$ $(\charact \Bbbk = 0)$
endowed with
a linear map $D:C\to C$ and a family of linear maps
$\oo{n}: C\otimes C \to C$ satisfying the following axioms:

\begin{itemize}
\item[]{\rm(C1)} for any $a,b\in C$ there exists $N=N(a,b)$
such that $a\oo{n} b=0$ for all $n\ge N$;
\item[]{\rm(C2)} $Da\oo{n} b = -n a\oo{n-1} b$;
\item[]{\rm(C3)} $a\oo{n}Db = D(a\oo{n} b) + na\oo{n-1} b$.
\end{itemize}

\noindent
If $C$ is a finitely generated $\Bbbk[D]$-module then
$C$ is said to be a conformal algebra of {\em finite type}
(or {\em finite\/} conformal algebra).
\end{defn}

Axiom (C1) allows to define the so-called locality function
$N_C: C\times C \to \Zset_+$,
\[
 N_C(x,y) = \min\{N\in \Zset_+ \mid x\oo{n}y =0\ \mbox{for all}\ n\ge N\}.
\]

A conformal algebra $C$ is said to be
$\Zset_2$-graded
if $C=C_0\oplus C_1$ as a $\Bbbk[D]$-module and
$C_i\oo{n} C_j \subseteq C_{(i+j)\mod 2}$.
By $p(a)$ we denote the parity of $a\in C$:
$p(a)=i$ if $a\in C_i$, $i=0,1$.

For any conformal algebra
$C$ there exists an ordinary (non-associative, in general) algebra
$A$ such that $C$ can be embedded into
the space of formal power series
$A [[z,z^{-1}]]$, where $D=\partial_z$ and
the $\oo{n}$-products on $A[[z,z^{-1}]]$ are given by
\[
a(z)\oo{n} b(z) = \Ress_{w=0} a(w)b(z)(w-z)^n,
\quad n\in \Zset_+,
\]
where $\Ress_{w=0} f(z,w)$ stands for the coefficient of $w^{-1}$
in $f(z,w)$.
Such an algebra $A$ is not unique, but there exists a universal one
denoted by $\Coeff C$.
Namely \citep{K2,Ro1},
$\Coeff C = \Bbbk[t,t^{-1}]\otimes _{\Bbbk[D]} C $
as a linear space (here we consider $\Bbbk[t,t^{-1}]$
as a right $\Bbbk[D]$-module thinking of $D$ as of $-\frac{d}{dt}$).
Let us write $a(n)$ for $t^n\otimes _{\Bbbk[D]}a$, $a\in C$, $n\in \Zset$.
The multiplication on $\Coeff C$ is well-defined by
\[
 a(n)b(m) = \sum\limits_{s\ge 0} \binom{n}{s} (a\oo{n-s} b)(m+s).
\]
The algebra $\Coeff C$ is called the
{\em coefficient algebra\/} of~$C$.

There is a correspondence between identities on $\Coeff C$
and conformal identities on~$C$.
In particular, $\Coeff C$
is associative if and only if $C$ satisfies
\begin{equation}\label{assoc-l}
(a\oo{n} b)\oo{m} c = \sum\limits_{s\ge 0} (-1)^s
 \binom{n}{s}
a\oo{n-s} (b\oo{m+s} c), \quad a,b,c\in C, \ n,m\in \Zset_+.
\end{equation}
The system of relations (\ref{assoc-l}) is equivalent to
\begin{equation}\label{assoc-r}
a\oo{n} (b\oo{m} c) = \sum\limits_{s\ge 0}
 \binom {n}{s} (a\oo{n-s} b)\oo{m+s} c,
  \quad a,b,c\in C, \ n,m\in \Zset_+.
\end{equation}

If $C$ is $\Zset_2$-graded then $\Coeff C$ inherits the grading:
$p(a(n))=p(a)$. Coefficient algebra $\Coeff C$ is a Lie superalgebra if and only if
$C$ satisfies
\begin{eqnarray}
& a\oo{n} b + (-1)^{p(a)p(b)} \{a\oo{n} b\}=0,
                                       \label{anticomm} \\
& a\oo{n}(b\oo{m} c) - (-1)^{p(a)p(b)} b\oo{m}(a\oo{n} c) =
 \sum\limits_{s\ge 0}
  \binom{n}{s} (a\oo{n-s} b)\oo{m+s} c,
                                    \label{Jacobi}
\end{eqnarray}
where $\{ b\oo{n} a\} = \sum\limits_{s\ge 0}
\frac{(-1)^{n+s}}{s!} D^{s}(b\oo{n+s} a)$,
$n,m\in \Zset_+$.

A conformal algebra $C$ is called associative, if   $\Coeff C$
is associative, i.e., if $C$ satisfies (\ref{assoc-l})
or (\ref{assoc-r}).
Analogously, $C$ is called Lie conformal superalgebra,
if $\Coeff C$ is a Lie superalgebra, i.e., if $C$ satisfies
(\ref{anticomm}) and (\ref{Jacobi}).
In order to distinguish notations, we will denote conformal products
in associative conformal algebras by $(\cdot\oo{n}\cdot)$
and in Lie conformal algebras by $(\cdot\so{n}\cdot)$,
$n\in {\Zset_+}$.

\begin{prop}[e.g. Kac, 1999]\label{prop adjoint}
Let $C$ be an associative $\Zset_2$-graded conformal algebra. Then the
same $\Bbbk[D]$-module $C$ endowed with new operations
$a\so{n} b = [a\oo{n} b]$, where
\begin{equation}\label{commutator}
[a\oo{n}b]= a\oo{n} b-(-1)^{p(a)p(b)}\{b\oo{n} a\},
   \quad a,b\in C,\ n\in \Zset_+,
\end{equation}
is a Lie conformal superalgebra denoted by $C^{(-)}$.
\end{prop}

\begin{defn}[Kac, 1997]
Let $V$ be a left (unital) $\Bbbk[D]$-module. A {\em conformal
endomorphism\/} is a $\Bbbk$-linear map
$a: \Bbbk[D] \to \End_{\Bbbk} V$ such that
\begin{itemize}
\item[]{\rm(i)}
 $\codim \{h\in \Bbbk[D] \mid a(h)v=0\} <\infty$ for any $v\in V$;
\item[]{\rm(ii)}
 $a(h)Dv = Da(h)v + a(h')v$, $h'$ is the ordinary derivative of~$h$.
\end{itemize}
\end{defn}

Let $\Cend V$ denotes the set of all conformal endomorphisms.
One may define operations $D$ and $(\cdot\oo{n}\cdot)$,
$n\in \Zset_+$, on $\Cend V$ as follows:
\begin{eqnarray}
& (Da)(h) = -a(h'),\nonumber \\
& (a\oo{n} b)(h) = \sum\limits_{s\ge 0} (-1)^s \binom{n}{s}
 a(D^{n-s})b(D^{s}h).\nonumber
\end{eqnarray}
Then (C2), (C3) and (\ref{assoc-l}) hold. If $V$ is a
finitely generated $\Bbbk[D]$-module then (C1) also holds,
so $\Cend V$ turns into an associative conformal algebra.
If $V$ is a free $N$-generated $\Bbbk[D]$-module then
$\Cend V$ is denoted by $\Cend_N$. The structure of this algebra was
particulary considered in \citep{BKL1}.

\begin{defn}[Cheng et al., 1997b; Kac, 1999]\label{defn conf_mod}
Let $C$ be a Lie conformal
superalgebra.
A {\em representation\/} of $C$ on a $\Bbbk[D]$-module $V$
is a linear map $\rho : L\to \Cend V$ such that
\begin{eqnarray}
& \rho(Da)=D\rho(a),\nonumber \\
& \rho (a\so{n} b) = (\rho(a)\oo{n}\rho(b)) - (-1)^{p(a)p(b)} \{\rho(b)\oo{n}\rho(a)\}
\end{eqnarray}
for all $a,b\in L$, $n\in \Zset_+$.
If $V$ is a finitely generated $\Bbbk[D]$-module then the representation $\rho $
is said to be {\em finite}.
\end{defn}

\section{Free conformal algebras and the Composition-Diamond lemma}

The study of free conformal algebras was initiated in \citep{Ro1},
where free associative and Lie conformal algebras were
constructed.

For a set of generators $B$ and a locality function
$N:B\times B \to \Zset_+$ there exists an associative (Lie) conformal
algebra $F_N(B)$ such that for any associative (resp., Lie) conformal
algebra $C$ and for any map $\iota: B\to C$ such that
$N_C(\iota(a),\iota(b))\le N(a,b)$, $a,b\in B$,
there exists a unique homomorphism $\varphi :F_N(B)\to C$
such that $\varphi (a)=\iota(a)$, $a\in B$.

Let us present the construction of the free associative conformal
algebra with a constant locality function~$N$ (c.f. \cite{BFK1}).
Consider the (ordinary) free associative algebra $\Bbbk\langle v,B\rangle $,
where $v$ is a formal variable, $v\notin B$.
Free $\Bbbk[D]$-module
$F(B)=\Bbbk[D]\otimes \Bbbk\langle v,B\rangle $
can be endowed with conformal products by setting
\[
(1\otimes f)\oo{n}(1\otimes g) = 1\otimes f\frac{\partial^n g }{\partial v^n},
\]
for $f,g\in \Bbbk\langle v,B\rangle $, and then by making use of (C2), (C3).
These operations turn $F(B)$ into an associative conformal algebra.
Conformal subalgebra of $F(B)$ generated by
$\{v^{N-1}a \mid a\in B\}$ is isomorphic to $F_N(B)$.

The following monomials (normal words)
\begin{eqnarray}\label{norm_w}
& w=D^s (a_1\oo{n_1} (a_2 \oo{n_2} \dots \oo{n_{k-1}}(a_k\oo{n_k} a_{k+1})\dots )),\\
& s \ge 0,\quad  a_i\in B,\quad  0\le n_i<N,\nonumber
\end{eqnarray}
form a linear basis of $F_N(B)$.
Linear combinations of normal words are called conformal polynomials.

Assume $w$ to be a normal word (\ref{norm_w}). By $\wt(w)$
we denote the following string:
\begin{equation}\label{weight}
 \wt(w) = (n_1,a_1,\dots ,a_k, n_k, a_{k+1},s).
\end{equation}
If $B$ is endowed with a linear order $\le $
such that $(B,\le )$ is a well-ordered set then we can expand this order to
normal words by the rule
\begin{equation}\label{ord}
 v\le w \iff \wt(v)\le \wt(w),
\end{equation}
comparing strings (\ref{weight}) by their length first and then
lexicographically. For a conformal polynomial $f\in F_N(B)$
denote by $\bar f$ its principal word:
\[
 f =\alpha \bar f + \sum\limits_{i} \alpha _iu_i,
    \quad \alpha \ne 0, \ u_i<\bar f.
\]

Every associative conformal algebra $C$ generated by $B$ such that
$N_C(a,b)\le N$, $a,b\in B$, is isomorphic to the quotient
algebra
$F_N(B)/I$ for some ideal $I$.
As usual, a set $S\subseteq F_N(B)$
generating $I$ as an ideal of $F_N(B)$ is called
a set of defining relations of $C$.
There is a natural problem: given a set $S$ of defining relations of $C$,
how to decide whether two conformal polynomials are equal in $C$?
In general, this problem is algorithmically unsolvable \citep{BFK1},
but there is a generalized (infinite) algorithm to treat it,
somewhat similar to the one of
\citep{Sh62,Buch,B76,KL}.

Let $S\subseteq F_N(B)$ be a set of conformal polynomials.
A normal word $w$ is said to be $S$-reduced if
$w$ cannot be presented as a principal part of
\[
 D^s(u\oo{n} f\oo{m} v) \quad \mbox{or}\quad  D^s(u\oo{n}g),
\]
where $f,g\in S$, $f$ is a $D$-free polynomial, $u$ and $v$ are
normal words, $0\le n,m<N$, $s\ge 0$.

In \citep{BFK1,BFK2}, the notion of a composition $(f,g)_w$
of conformal polynomials $f$, $g$ was introduced. In general, there are six
types of compositions of such polynomials. A set
$S\subseteq F_N(B)$ is called a Gr\"obner--Shirshov basis (GSB) if it is
closed under all compositions.

\begin{thm}[\cite{BFK2}]\label{CDlemma}
Let $S$ be a set of defining relations of an associative conformal algebra
$C$. If $S$ is a GSB then $S$-reduced normal words form a linear basis
of $C$. The converse is true if $S$ consists of $D$-free polynomials.
\end{thm}

\section{Associative envelopes of Lie conformal superalgebras}

Let $L$ be a conformal Lie superalgebra with operations
$D$ and $(\cdot\so{n}\cdot)$,
$n\in {\Zset_+}$.

\begin{defn}
An associative envelope of $L$ is a pair
 $(A,\varphi)$, where $A$ is an associative conformal algebra,
$\varphi: L\to A$ is a $D$-invariant linear map such that
\[
\varphi(a\so{n}b)
= \varphi (a)\oo{n}\varphi (b) - (-1)^{p(a)p(b)} \{\varphi (b)\oo{n}\varphi (a)\},
\quad a,b\in L,\ n\in \Zset_+,
\]
and $A$ is generated by $\varphi(L)$ as a conformal algebra.
\end{defn}

Note that $\varphi $ is not necessarily injective.

Two associative envelopes
$(A_1, \varphi _1)$, $(A_2, \varphi _2)$ of $L$ are isomorphic
if there exists an isomorphism $\psi: A_2\to A_1$ of associative
conformal algebras
such that
$\psi\varphi _2 = \varphi _1$.

The set ${\mathcal E}(L)$ of isomorphism classes of associative envelopes
can be ordered in the usual way:
$(A_1, \varphi _1) \preceq (A_2, \varphi _2)$ if there exists a homomorphism
$\psi : A_2\to A_1$
such that
$\psi\varphi _2 = \varphi _1$.

In contrast to the case of ordinary algebras, the
partially ordered set $(\mathcal E(L), \preceq )$
has no greatest element.
The main reason is the requirement of locality of elements $\varphi(L)$ in an
associative envelope $A$. However, there is a way to fix this problem \citep{Ro2}.

Let $B$ be a set of generators of $L$, and let
$N: B\times B \to {\Zset}_+$ be a fixed function.
Denote by ${\mathcal E}_N(L,B) $ the set of all associative envelopes
$(A,\varphi) $ of $L$ such that
$N_A(\varphi(a),\varphi(b)) \le N(a,b)$ for all $a,b\in B$.
Then ${\mathcal E}_N(L,B) $ has the greatest element (the universal
associative envelope with respect to generators $B$ and locality $N$)
denoted by $(U_N(L,B),\iota_N(L,B))$ or just $U_N(L,B)$, for short.

Let us state here the construction of $U_{N}(L,B)$.
Consider the coefficient algebra
$\mathcal L=\Coeff L$, this is a Lie
superalgebra generated by $\{b(n)\mid b\in B,\, n\in \Zset\}$.
Denote by $\mathcal I_N(B)$ the ideal of $U(\mathcal L)$ generated by
\begin{equation}\label{N-def}
\sum\limits_{s\ge 0} (-1)^s \binom{N(a,b)}{s}
  a(n-s)b(m+s),\quad  a,b\in B,\ n,m\in \Zset.
\end{equation}
Then formal power series
$\bar a(z)= \sum\limits_{n\in \Zset} (a(n)+\mathcal I_N(B)) z^{-n-1}
 \in U(\mathcal L)/\mathcal I_N(B)[[z,z^{-1}]] $
are pairwise mutually local, therefore, they generate an associative
conformal algebra that is $U_N(L,B)$.

Another way to construct $U_N(L,B)$ is to use a
presentation of $L$ by generators $B$ and defining relations $S_{Lie}$ \citep{BFK2}.
Consider the ideal $I_N(B)$ of $F_N(B)$ generated by
$S$, where $S$ is obtained from
$S_{Lie}$ by rewriting $(\cdot\so{n}\cdot) = [\cdot \oo{n} \cdot]$ via (\ref{commutator}).
Then  $U_N(L,B)\simeq F_N(B)/I_N(B)$.

Note that if $B$ consists of homogeneous elements of $L$ then
$U_N(L,B)$ inherits the grading, so
$\iota_N(L,B):L\to U_N(L,B)^{(-)}$
is a homomorphism of Lie conformal superalgebras.

A {\em superinvolution\/} of a $\Zset_2$-graded
conformal algebra $C$
is a $\Bbbk[D]$-linear map
$\sigma :C \to C$ such that
$p(\sigma (a))=p(a)$, $\sigma ^2=\Idd_C$,
$\sigma (a\oo{n} b) = (-1)^{p(a)p(b)} \{\sigma (b)\oo{n} \sigma (a)\}$,
$a,b\in C$.

\begin{lem}\label{lem lifting}
Let $L$ be a Lie conformal superalgebra generated by a subset $B$
of homogeneous elements.
Superinvolution $\sigma: L\to L$, $x\mapsto -x$,
can be expanded to $U_N(L,B)$
if and only if $N(a,b)=N(b,a)$ for all $a,b\in B$.
\end{lem}

\begin{proof}
Consider the canonical antipode $S : U(\mathcal L) \to U(\mathcal L)$,
$\mathcal L=\Coeff L$, $S(x)=-x$ for $x\in \mathcal L$.
If $N(a,b)=N(b,a)$ for any $a,b\in B$ then the relation
(\ref{N-def}) holds under $S$, so
$S$ induces a superinvolution of
$U(\mathcal L)/\mathcal I_N(B)$ that can be expanded to the
conformal algebra $U_N(L,B)$.
\end{proof}

Any associative envelope $(A,\varphi)\in {\mathcal E}_N(L,B) $
is a homomorphic image of $U_N(L,B)$. Therefore, it is interesting to
explore the cases when $U_N(L,B)$ is
a simple conformal algebra.

This case seems to be interesting by one more reason. For a fixed
set of generators $B$ one may consider the lattice $\mathcal{UE}
(L,B)$ of universal envelopes $U_N(L,B)$ as a subset of $\mathcal
E(L)$. Assume $L^2=L$ (e.g., $L$ is simple). Then the lowest point
of the lattice $\mathcal{UE}(L,B)$ is $\{0\}$. The set of minimal
(nonzero) points of this lattice necessarily includes all simple
universal envelopes. If we extend the set of generators, i.e.,
consider $B'\supset B$, then $\mathcal{UE}(L,B)\subseteq
\mathcal{UE}(L,B')$. Some of minimal points of the lattice
$\mathcal{UE}(L,B)$ may not be minimal in $\mathcal{UE}(L,B')$.
But simple universal envelopes of $\mathcal{UE}(L,B)$ are always
minimal in $\mathcal {UE}(L,B')$.

Let $L$ be a finitely generated conformal Lie superalgebra.
Any simple associative envelope of $L$ of at most linear growth defines
an irreducible finite representation of $L$. Indeed, it was shown in
\citep{Re1,Ko3} that every simple finitely generated associative conformal algebra
of at most linear growth can be embedded into $\Cend V$, $\rank V<\infty $,
as an irreducible subalgebra.

Conversely, let $\rho $ be a representation of a conformal superalgebra $L$
on a finite module $V$. Denote by $A_V(L)$ the associative conformal subalgebra
generated by $\rho(L)$ in $\Cend V$.
 The pair $(A_V(L), \rho)$ is an associative
envelope of $L$; two representations are equivalent if and only if the
corresponding envelopes are isomorphic.
If $\rho $ is irreducible then the associative conformal algebra $A_V(L)$
generated in $\Cend V$
by $\rho (L)$ acts irreducibly on~$V$.
Irreducible subalgebras of $\Cend V$ were completely described in \citep{Ko4}.
In particular, $A_V(L)$ is a simple conformal algebra of at most linear growth.

A finite representation $\rho $ induces a locality function $N$ on $L$, therefore,
 on a set
of generators $B\subset L$. Namely,
$N(x,y) = N_{A_V(L)} (\rho(x),\rho(y))$, $x,y\in B$.
The envelope $A_V(L)$ is a homomorphic image of the corresponding  universal envelope
$U_N(L,B)$.
Therefore, those finite irreducible representations that appear from
$A_V(L)\simeq U_N(L,B)$ are in some sense simplest ones.

\begin{defn}
Let $L$ be a conformal Lie superalgebra generated by a subset $B$.
An irreducible finite representation
$\rho : L \to \Cend V$ is called {\em universally defined\/}
if $A= A_V(L)\simeq U_N(L,B)$, where $N:B\times B\to {\Zset}_+$
is the locality function induced   by $\rho$, i.e.,
$N(a,b):=N_A(\rho(a),\rho(b))$.
\end{defn}

Note that this property depends on the choice of generating set $B$. However,
if a representation $\rho $ is universally defined with respect to $B$ then
so is $\rho $ with respect to any $B'\supseteq B$.

\begin{exmp}
Consider $L=\Curr \sll_2$ (current conformal algebra over $\sll_2$), and let $B=\{1\otimes e,1\otimes f,1\otimes h\}$,
where $e,f,h$ is the standard basis of $\sll_2$.
There are no universally defined representations of $L$
with respect to~$B$.
\end{exmp}

\begin{exmp}\label{exmp-Vir}
For $L=\Virr=\Bbbk[D]v$ (Virasoro conformal algebra), $B=\{v\}$,
a universally defined representation exists and unique. Namely,
if $N(v,v)=2$ then $U_N(L,B)\simeq \Cend_{1,v}$ as it was actually shown in
\citep{BFK1}.
\end{exmp}

\section{Universally defined representations of $W_n$}

In this section we describe universally defined representations
of ${\Zset}_2$-graded extensions $W_n$, $n\ge 0$, of the Virasoro conformal algebra.
The explicit construction of these conformal Lie superalgebras via formal power series
is stated, for example, in \citep{FK}. It is easy to show that one may present $W_n$
by generators and defining relations as follows.

\begin{prop}[Kolesnikov, 2004]\label{prop-defining}
Conformal superalgebra $W_n$ is generated
by the set
$B=\{v, \xi _i, \partial _i \mid i =1,\dots ,n\}$
with defining relations
\begin{equation}\label{eqn-def.rel}
\begin{array}{c}
\xi _i \so{0} \xi _j = - \xi _j \so{0} \xi _i, \quad
\partial _i \so{n} \partial _j = 0,\ n\ge 0, \\
\partial _j\so{0} \xi _i = \delta _{ij}v,
\quad
v \so{0} \xi _i = \xi _i \so{0} v = D\xi _i, \quad
\xi _i \so{1} v = 2\xi _i, \quad \partial _j \so{0} v = 0,  \\
\partial _j \so{1} v = \partial _j,\quad
\xi _i \so{0} \xi _j = \frac{1}{2} D(\xi _i\so{1} \xi_j), \quad
v\so{0} v = Dv, \quad  v\so{1} v = 2v, \\
\xi_i\so{n}\xi_j=
 v\so{n}\partial_j=\xi_i\so{n}\partial_j = v\so{n}\xi_i=v\so{n}v =0, \ n\ge 2, \\
\end{array}
\end{equation}
(here $p(v)=0$, $p(\xi _i)=p(\partial _i)=1$).
\end{prop}

Let $A_n$ be the (ordinary) associative algebra with a unit generated by
the set $\{\xi_i,\partial_i\mid i=1,\dots,n\}$ with the following relations:
\begin{eqnarray}\nonumber
& \xi_i\xi_j + \xi_j\xi_i=0, \quad \partial_i\partial_j + \partial_j\partial_i = 0, \\
& \partial_i\xi_j + \xi_j\partial_i =\delta_{ij}. \nonumber
\end{eqnarray}
We may consider  $\Bbbk[D]\otimes A_n[v]$
as an associative conformal algebra with operations
\[
(1\otimes a(v)) \oo{n}(1\otimes b(v))= 1\otimes a(v)\frac{\partial^n b(v)}{\partial v^n},
\]
$a(v), b(v)\in A_n[v]$.
Since $A_n \simeq M_{2^n}(\Bbbk)$,
the associative conformal algebra obtained is isomorphic
to $\Cend_{2^n}$ \citep{BKL1,K2,Re1}.
This algebra is ${\Zset}_2$-graded with respect
to the usual grading on $A_n[v]$
($p(v)=0$, $p(\xi_i) = p(\partial _i) =1$).

\begin{prop}\label{prop-embedding}
The following maps define homomorphisms of conformal Lie superalgebras
$W_n \to (\Bbbk[D]\otimes A_n[v])^{(-)}$:
\begin{eqnarray}
&& \varphi_1: v \mapsto v-D,\ \xi_i \mapsto (v-D)\xi_i,\ \partial_i\mapsto \partial_i;
   \label{eqn-1emb} \\
&& \varphi_2: v \mapsto v,\ \xi_i \mapsto v\xi_i,\ \partial_i\mapsto \partial_i.
   \label{eqn-2emb}
\end{eqnarray}
\end{prop}

\begin{proof}
It is enough to check that the relations (\ref{eqn-def.rel}) hold under $\varphi_k$,
$k=1,2$. The computation is straightforward.
\end{proof}

Denote by $C_k$ ($k=1,2$) the associative conformal subalgebra of $\Bbbk[D]\otimes A_n[v]$
generated by $\varphi_k(W_n)$. To write down the explicit form of these algebras,
let us fix an isomorphism
$\Bbbk[D]\otimes A_n[v]\to \Cend_{2^n}$ as follows.
If we identify the Grassman algebra $\wedge_n(\xi_1,\dots,\xi_n)$ with the
$2^n$-dimensional vector space over $\Bbbk $ then $A_n$ turns into the full
algebra of linear transformations of this space.
Let us fix a linear basis $(e_1,\dots,e_{2^n})$
of $\wedge_n$ in such a way that $e_1=1$, $e_{2^n}= \xi_1\dots\xi_n$
and identify $A_n\simeq \End \wedge_n$ with $M_{2^n}(\Bbbk )$.

Then
$C_1$ maps onto $\Cend_{2^n,Q}\simeq M_{2^n}(\Bbbk[D,v])Q(v-D)$, where
$Q(v)=\diag(v,1,\dots, 1)$. Analogously,
conformal algebra
$C_2$ can be identified
with $\Cend_{\widetilde{Q},2^n}\simeq \widetilde{Q}(v)M_{2^n}(\Bbbk[D,v])$,
where
$\widetilde Q = \diag(1,\dots ,1,v)$.

Let us compute the locality functions
$N_k:B\times B \to {\Zset}_+$, $k=1,2$,
where
$N_k(x,y)=N_{C_k}(\varphi_k(x), \varphi_k(y))$,
$x,y\in B$.

\medskip
\centerline{\hfil $N_1(x,y)$\hfil\hfil\hfil$N_2(x,y)$\hfil}
\nobreak
\centerline{\def\tabrule{\multispan{17}\hrulefill\cr}
\def\tabvskip#1{height #1pt& && && && && && && && &\cr}
\hfil\vbox{\offinterlineskip\tabskip=0pt plus 1fill
\halign{\vrule#&\ \hfil #\hfil\ &\vrule #&\ \hfil #\hfil\ &\vrule #&
\ \hfil #\hfil\ &#&\ \hfil #\hfil\ &#&
\ \hfil #\hfil\ &\vrule #&\ \hfil #\hfil\ &#&
\ \hfil #\hfil\ &#&\ \hfil #\hfil\ &\vrule #\cr
\tabrule
height 4pt& && &\omit& && && &\omit& && && &\cr
& \smash{\lower 10pt\hbox{$x$}}&& &\omit & && && $y$&\omit & && && &\cr
height 4pt& && &\omit& && && &\omit& && && &\cr
& &&\multispan{14}\hrulefill\cr
\tabvskip4
& && $v$&& $\xi_1$&& $\dots$&& $\xi_n$&& $\partial_1$&& $\dots$&& $\partial_n$&\cr
\tabvskip4
\tabrule
\tabvskip4
& $v$&& 2&& 2&&\dots && 2&& 2&& \dots && 2&\cr
\tabvskip4
\tabrule
\tabvskip4
& $\xi_1$&& 2&& 0&&\dots && 2&& 2&& $\dots$&& 2&\cr
& $\vdots $&& $\vdots$&& $\vdots$&& $\ddots$&& $\vdots$&& $\vdots$&& $\ddots$&& $\vdots$&\cr
\tabvskip4
& $\xi_n$&& 2&& 2&&\dots && 0&& 2&& $\dots$&& 2&\cr
\tabvskip4
\tabrule
\tabvskip4
& $\partial_1$&& 1&& 1&&\dots && 1&& 0&& $\dots$&&  1&\cr
& $\vdots $&& $\vdots$&& $\vdots$&& $\ddots$&& $\vdots$&& $\vdots$&& $\ddots$&& $\vdots$&\cr
\tabvskip4
& $\partial_n$&& 1&& 1&&\dots && 1&& 1&& $\dots$&& 0&\cr
\tabvskip4
\tabrule
}}\hfil\hfil
\vbox{\offinterlineskip\tabskip=0pt plus 1fill
\halign{\vrule#&\ \hfil #\hfil\ &\vrule #&\ \hfil #\hfil\ &\vrule #&
\ \hfil #\hfil\ &#&\ \hfil #\hfil\ &#&
\ \hfil #\hfil\ &\vrule #&\ \hfil #\hfil\ &#&
\ \hfil #\hfil\ &#&\ \hfil #\hfil\ &\vrule #\cr
\tabrule
height 4pt& && &\omit& && && &\omit& && && &\cr
& \smash{\lower 10pt\hbox{$x$}}&& &\omit & && && $y$&\omit & && && &\cr
height 4pt& && &\omit& && && &\omit& && && &\cr
& &&\multispan{14}\hrulefill\cr
\tabvskip4
& && $v$&& $\xi_1$&& $\dots$&& $\xi_n$&& $\partial_1$&& $\dots$&& $\partial_n$&\cr
\tabvskip4
\tabrule
\tabvskip4
& $v$&& 2&& 2&&\dots && 2&& 1&& \dots && 1&\cr
\tabvskip4
\tabrule
\tabvskip4
& $\xi_1$&& 2&& 0&&\dots && 2&& 1&& $\dots$&& 1&\cr
& $\vdots $&& $\vdots$&& $\vdots$&& $\ddots$&& $\vdots$&& $\vdots$&& $\ddots$&& $\vdots$&\cr
\tabvskip4
& $\xi_n$&& 2&& 2&&\dots && 0&& 1&& $\dots$&& 1&\cr
\tabvskip4
\tabrule
\tabvskip4
& $\partial_1$&& 2&& 2&&\dots && 2&& 0&& $\dots$&&  1&\cr
& $\vdots $&& $\vdots$&& $\vdots$&& $\ddots$&& $\vdots$&& $\vdots$&& $\ddots$&& $\vdots$&\cr
\tabvskip4
& $\partial_n$&& 2&& 2&&\dots && 2&& 1&& $\dots$&& 0&\cr
\tabvskip4
\tabrule
}}\hfil}
\medskip

Although conformal algebras $C_1$ and $C_2$ are isomorphic,
the associative envelopes
$(C_1,\varphi_1)$ and $(C_2,\varphi_2)$ are not isomorphic for $n>0$
(hence,  the corresponding representations are not equivalent).
The reason is that $N_1(v,\partial_i) = 2\neq N_2(v,\partial_i)=1$.
For $n=0$ these envelopes are isomorphic: they correspond to the universally
defined representation of the Virasoro conformal algebra from Example~\ref{exmp-Vir}.

\begin{thm}\label{thm-univ-W}
For $n>0$ there exist exactly two universally defined representations
of $W_n$ with respect to $B=\{v,\xi_i,\partial_i\mid i=1,\dots, n\}$.
Namely, these representations correspond to the associative envelopes
$(C_1,\varphi _1)$ and $(C_2,\varphi _2)$.
\end{thm}

\begin{proof}
In \citep{Ko2}, the Gr\"obner--Shirshov basis $S_1$ of $U_{N_1}(W_n,B)$ was found.
The set of $S_1$-reduced conformal words consists of
\begin{eqnarray}
& D^t ((v\oo{0})^n \xi _{i_1} \oo{0}
\dots  \oo{0} \xi _{i_s} \oo{0} \partial _{j_1}\oo{0}
\dots \oo{0} \partial _{j_q}),
\quad     n>0,\ s,q\ge 0,
                                                      \label{red-1-1}\\
& D^t (\xi _{i_1} \oo{0}\dots  \oo{0}\xi _{i_{r+1}}\oo{1} \dots \oo{1} \xi _{i_s}
\oo{1} \partial _{j_1}\oo{0} \dots \oo{0}
\partial _{j_q}),
\quad  1\le r\le s,\ q\ge 0,
                                                      \label{red-1-2}\\
& D^t (\xi _{i_1}\oo{1}\dots \oo{1}\xi _{i_s}\oo{1}\partial _{j_1}\oo{0} \dots \oo{0} \partial _{j_q}),
\quad   s\ge 0, \ q>0,
                                                      \label{red-1-3}
\end{eqnarray}
where
$1\le i_1<\dots <i_s\le n$, $1\le j_1<\dots <j_q\le n$, $t\ge 0$;
by default,  we assume the bracketing is right-justified.
Here we use the following order on $B$:
$v<\xi_1<\dots<\xi_n<\partial_1<\dots <\partial_n$.

The map $\varphi _1: W_n \to C_1$ can be extended to a
homomorphism $U_{N_1}(W_n,B) \to C_1$ because of the choice
of $N_1$. Let us also denote this homomorphism by $\varphi _1$.
It is easy to see that the images of
(\ref{red-1-1})--(\ref{red-1-3}) are linearly independent in $C_1$ \citep{Ko2},
so $(C_1,\varphi _1)$ is the universal associative
envelope corresponding to the locality function $N_1$ on~$B$.
Therefore, $\varphi _1$ is a universally defined representation
with respect to~$B$.

Now, let us show that $C_2\simeq U_{N_2}(W_n,B)$.
The initial set of defining relations of $U_{N_2}(W_n,B)$
appears from (\ref{eqn-def.rel}):
\begin{eqnarray}
& \partial _i\oo{0} \xi _j + \xi _j\oo{0} \partial _i = \delta _{ij}v,
                                      \label{eq-2-1}\\
& 2(\xi _i\oo{0} \xi _j + \xi _j\oo{0}\xi _i)
 = D(\xi _i\oo{1}\xi _j + \xi _j\oo{1}\xi _i), \quad i\ne j,
                                      \label{eq-2-2}\\
& \partial _i\oo{0}\partial _j + \partial _j\oo{0} \partial _i = 0,
i\ne j,
                                      \label{eq-2-3}\\
& v\oo{0}\xi _i - \xi _i\oo{0} v + D(\xi _i \oo{1} v) = D\xi _i,
                                      \label{eq-2-4}\\
& \xi _i\oo{0} v - v\oo{0} \xi _i + D(v\oo{1} \xi _i) = D\xi _i,
                                      \label{eq-2-5}\\
& \xi _i\oo{1} v + v\oo{1} \xi _i = 2\xi _i,
                                      \label{eq-2-6}\\
& \partial _i \oo{0} v - v\oo{0} \partial _i = 0,
                                      \label{eq-2-7}\\
& \partial _i\oo{1} v = \partial _i,
\quad  v\oo{1} v = v.
                                      \label{eq-2-8}
\end{eqnarray}

\begin{lem}
The following relations hold on $U_{N_2}(W_n,B)$:
\begin{eqnarray}
& \xi _i\oo{0} \xi _j = -\xi _j\oo{0} \xi _i,
\quad \xi _i\oo{1}\xi _j = -\xi _j\oo{1} \xi _i,
                                      \label{eq-A-1}\\
& v\oo{0}\xi _i = \xi _i\oo{0} v,
\quad  \xi _i\oo{1} v= \xi _i,\quad  v\oo{1}\xi _i = \xi _i,
                                      \label{eq-A-2}\\
& \xi _i\oo{1} (\xi _j\oo{0} \xi _k)= 2\xi _i\oo{0} (\xi _j\oo{1}\xi _k),
\quad i<j<k,
                                      \label{eq-A-3}\\
& \xi _i \oo{1} (\xi _j\oo{0} v) = 2\xi _i\oo{0} \xi _j,
\quad  i<j,
                                      \label{eq-A-4}\\
& \partial _i\oo{1} (\xi _j \oo{0} v) = 2\partial _i\oo{0} \xi _j,
                                      \label{eq-A-5}\\
& \partial _k\oo{1} (\xi _i\oo{0} \xi _j) = 2\partial _k(\xi _i\oo{1} \xi _j),
\quad
i<j.
                                      \label{eq-A-6}
\end{eqnarray}
\end{lem}

\begin{proof}
To deduce the required relations, we are going to perform
the Buchberger--Shirshov algorithm
for conformal algebras \citep{BFK2} starting with relations
(\ref{eq-2-1})--(\ref{eq-2-8}). Define the order of conformal
monomials as in (\ref{ord}) assuming
$v>\xi _n>\dots >\xi _1 > \partial _n>\dots >\partial _1$.

Consider (\ref{eq-2-1}) for $i=j$ and multiply with
$\xi _j\oo{0} $ and $\oo{0}\xi _j$ to obtain
$\xi _j\oo{0} v = \xi _j\oo{0} \partial _j \oo{0} \xi _j = v\oo{0}\xi _j$.

Multiplying (\ref{eq-2-6}) with $\oo{1}v$
and applying (\ref{eq-2-8}) we obtain
$v\oo{1}(\xi _i\oo{1} v) = v\oo{1} \xi _i$.
The same relations allow
to compute the left-hand side:
$v\oo{1} (\xi_i\oo{1} v) = (v\oo{1} \xi _i) \oo{1} v
= -(\xi _1\oo{1} v)\oo{1} v + 2\xi _i\oo{1} v$.
Therefore, $\xi _i\oo{1} v = v\oo{1} \xi _i = \xi _i$,
and (\ref{eq-A-2}) proved.
To get the remaining relations (\ref{eq-A-1}), one can
multiply (\ref{eq-2-2}) with $\oo{1} v$ and $\oo{2} v$.

Relation (\ref{eq-A-4}) appears as the composition of intersection
$(f,g)_w$ \citep{BFK2}, where
$f = v\oo{0} \xi _j - \xi _j\oo{0} v$,
$g= \xi _i\oo{1} v - \xi _i$,
$w=\xi _i\oo{1}v\oo{0} \xi _j$.

To deduce (\ref{eq-A-3}), consider the composition of intersection
$(f,g)_w$, where
$f = \xi _i\oo{1} (\xi _j\oo{0} v) - 2\xi _i\oo{0} \xi _j$,
$g= v\oo{1} \xi _k - \xi _k$,
$w=\xi _i\oo{1} \xi _j\oo{0} v\oo{1}\xi _k$.

Relations (\ref{eq-A-5}) and (\ref{eq-A-6}) can be obtained
in a similar way.
\end{proof}

Let $S_2$ stands for the set of relations
(\ref{eq-2-1}),
(\ref{eq-2-3}),
(\ref{eq-2-6})--(\ref{eq-A-6}).
We do not need
(\ref{eq-2-2}), (\ref{eq-2-4}), (\ref{eq-2-5})
any more since these relations follow from
(\ref{eq-A-1}), (\ref{eq-A-2}).
The set of $S_2$-reduced normal words
 consists of
\begin{equation}                           \label{eq-red-2}
\begin{array}{c}
D^t (\partial _{j_1}\oo{0} \dots  \partial _{j_k}\oo{0}
\xi _{i_1}\oo{0} \dots  \oo{0} \xi _{i_s}\oo{0} (v\oo{0})^n \oo{0} v),
\quad
n\ge 0,\ s,k\ge 0,   \\
D^t(\partial _{j_1}\oo{0} \dots \oo{0} \partial _{j_k}\oo{0}
\xi _{i_1}\oo{0} \dots \oo{0} \xi _{i_r}\oo{1} \dots  \oo{1} \xi _{i_s}),
\quad
1\le r\le s,\ k\ge 0, \\
D^t(\partial _{j_1}\oo{0} \dots \oo{0} \partial _{j_k}
\oo{1}\xi _{i_1}\oo{1} \dots  \oo{1} \xi _{i_s}),
\quad  k> 0,\ s\ge 0,
\end{array}
\end{equation}
where $t\ge 0$,
$1\le j_1<\dots <j_k\le n$,
$1\le i_1<\dots <i_s\le n$.

There exists a homomorphism $U_{N_2}(W_n,B) \to C_2$
 extending $\varphi _2:W_n\to C_2$. Let us denote it also by $\varphi _2$.
It is easy to compute the images of (\ref{eq-red-2}) under
$\varphi _2$: these are
\begin{eqnarray}
& D^t \otimes \partial _{j_1}\dots \partial _{j_k}\xi _{i_1}\dots \xi _{i_s}v^{n+s},
\quad  n>0, \ s,k\ge 0,\nonumber
\\
& D^t \otimes  \partial _{j_1}\dots \partial _{j_k}\xi _{i_1}\dots \xi _{i_s}v^r,
\quad 1\le r\le s,\ k\ge 0,\nonumber
\\
& D^t \otimes  \partial _{j_1}\dots \partial _{j_k} \xi _{i_1}\dots \xi _{i_s},
\quad k>0,\ s\ge 0,\nonumber
\end{eqnarray}
respectively. The images obtained are linearly independent
in $\Bbbk[D]\otimes A_n[v]$, hence, the homomorphism
$\varphi _2: U_{N_2}(W_n,B) \to C_2$ is an isomorphism
of universal envelopes.

We have proved that the associative envelope $(C_2,\varphi _2)$ of
$W_n$ gives rise to a universally defined representation of $W_n$
which is not equivalent (for $n>0$) to the representation coming
from $(C_1,\varphi _1)$. Let us show that there are no other
universally defined representations with respect to the set of
generators $B$ (as well as to any greater set of generators
$B'\supseteq B$).

Assume that an associative envelope $(C,\varphi )$ of $W_n$
corresponds to a universally defined representation with
respect to~$B$, i.e., $C \simeq U_{N}(W_n,B)$,
where $N(x,y)=N_C(\varphi (x),\varphi (y))$, $x,y\in B$.

If there exists $k\in \{1,2\}$
such that $N_k(x,y)\le N (x,y)$ for all $x,y\in B$
then $C_k$ is a homomorphic image of $C$ (the homomorphism
would preserve $B$). Since $C$ is necessarily simple \citep{Ko4},
the associative envelope $(C,\varphi )$ should be isomorphic
to $(C_k,\varphi _k)$, so the representations $\varphi $ and
$\varphi _k$ are equivalent.

Hence, we have to assume that there exist
$x_1,y_1,x_2,y_2\in B$
such that
$N(x_1,y_1)< N_1(x_1,y_1)$ and $N(x_2,y_2)<N_2(x_2,y_2)$.
But $U_N(W_n,B)$ cannot be simple
for such $N$. To show that, it is sufficient to consider several cases.
Let us focus on some of these cases to show the technique.

Suppose $N(v,\xi _i)<2$ for some $i$.
Denote $\xi =\xi _i$, $\partial =\partial _i$ and
proceed as follows. Defining relations
(\ref{eqn-def.rel}) imply $\xi \oo{1} v = 2v$,
$\xi \oo{0} v - v\oo{0} \xi  = D\xi $.
For any $m\ge 0$ we have
\begin{eqnarray}\nonumber
 \xi \oo{m} \xi  &=& -\frac{1}{m+1} D\xi \oo{m+1} \xi
= -\frac{1}{m+1} (\xi \oo{0} v - v\oo{0} \xi )\oo{m+1} \xi    \\
&=& \frac{1}{m+1} v\oo{0} (\xi \oo{m+1} \xi ).\nonumber
\end{eqnarray}
Since $\xi \oo{m} \xi  = 0$ for a sufficiently large $m$,
we may conclude that $N(\xi ,\xi )=0$.
Now, multiply the defining relation
\[
 \partial \oo{0} \xi  + \xi \oo{0} \partial  - D(\xi \oo{1} \partial )
+\dots  = v
\]
with $\oo{0} \xi $ and $\xi \oo{0}$. Then
$v\oo{0}\xi =\xi \oo{0}\partial \oo{0} \xi  = \xi \oo{0} v$,
so $D\xi =0$ in $U_N(W_n,B)$. Hence, $\varphi :W_n\to C$
is not injective and $C=0$.

In the same way, one may get $C=0$ assuming that $N(v,\partial _i)<1$
for some $i\in\{1,\dots ,n\}$.

If there exist $i\ne j$ such that $N(\partial _i,\partial _j)<1$
then $N(\partial _i,\partial _j)=N(\partial _j,\partial _i)=0$.
Consider the relation
\[
 v= [\partial _i\oo{0}\xi _i] = \partial _i\oo{0} \xi _i
+ \{\xi _i\oo{0} \partial _i \}
\]
and multiply with $\oo{m} \partial _j$, $m\ge 0$.
Then $v\oo{m} \partial _j = \partial _i\oo{0} (\xi _i\oo{m} \partial _j)
= -\partial _i\oo{0} \{\partial _j\oo{m} \xi _i\} = 0$.
Hence, $N(v,\partial _j)=0$ but this case has already been explored.

Probably, the most difficult case is when
$N(\xi _i,\partial _j)<2$ and $N(\partial _l,\xi _k)<2$
for some $i,j,k,l\in \{1,\dots,n\}$.
Since $\xi _i\oo{m} \partial _j=0$ for $m\ge 1$, we have
\[
 0= \partial _i\oo{0} (\xi _i\oo{m} \partial _j)
 = (v-\{\xi_i\oo{0}\partial_i\})\oo{m} \partial_j
= v\oo{m} \partial _j - \xi _i\oo{m} (\partial _i\oo{0} \partial _j).
\]
Note that $\partial _i\oo{0}\partial _j + \{\partial _j\oo{0} \partial _i\}=0$,
so
\[
\xi _i\oo{m} (\partial _i\oo{0} \partial _j) =
 - \xi_i \oo{m} \{\partial _j\oo{0} \partial _i\}
=-\{(\xi _i\oo{m}\partial _j)\oo{0} \partial _i \}=0,
\]
therefore, $v\oo{m}\partial _j =0$ for any $m\ge 1$.

Relation $\partial _l\oo{m} \xi _k = 0$ ($m\ge 1$)
implies
\[
0=\{(\partial _l\oo{m}\xi _k)\oo{0} \partial _k\}
=\partial _l\oo{m}(v-\partial _k\oo{0} \xi _k).
\]
But
\[
\partial _l\oo{m} (\partial _k\oo{0} \xi _k)
= \{\partial _l\oo{0} \partial _k\}\oo{m} \xi _k
=-(\partial _k\oo{0} \partial _l)\oo{m} \xi _k = 0,
\]
so
$\partial _l\oo{m} v =0$ for any $m\ge 1$.

We have obtained $N(v,\partial _j)\le 1$,
$N(\partial _l, v)\le 1$. If $j=l$ then the result
is obvious; if $j\ne l$ then
$\partial _j\oo{1} v = \partial _j$,
$v\oo{1} \partial _l=\partial _l$,
$N(\partial _j,v)=2$,
and for any $m\ge 0$
\[
 \partial _l\oo{m} \partial _j
=\partial _l\oo{m} (\partial _j\oo{1} v)
=-\{\partial _j\oo{m}\partial _l\}\oo{1} v
=(-1)^{m+1} (\partial _j\oo{m+1} \partial _l)\oo{0} v.
\]
Hence, $N(\partial _j,\partial _l)=0$, but this case has already
been explored.

In the same way, all other choices of
$x_1,y_1,x_2,y_2\in B$ also lead to $C=0$.
Therefore, there are no more universally defined representations of $W_n$
with respect to~$B$.
\end{proof}

\begin{cor}
The set $S_2$ of relations
(\ref{eq-2-1}),
(\ref{eq-2-3}),
(\ref{eq-2-6})--(\ref{eq-A-6})
is a Gr\"obner--Shirshov basis of $U_{N_2}(W_n,B)$.
\end{cor}

\section{On universally defined representations of $K_n$}

Consider the linear map $\wedge_n\oplus \sum\limits_{j=1}^n \wedge_n\partial_j\to W_n$
defined by
\begin{eqnarray*}
1 && \mapsto v, \\
\partial_j && \mapsto \partial_j, \quad i=1,\dots, n, \\
\xi_I=\xi_{i_1}\dots \xi_{i_r} && \mapsto \frac{1}{2^{r-1}} \xi_{i_1}\so{1}\dots\so{1}
   \xi_{i_r}, \quad 1\le r\le n, \\
\xi_I\partial_j=\xi_{i_1}\dots \xi_{i_r}\partial_j && \mapsto
     \frac{1}{2^{r-1}} (\xi_{i_1}\so{1}\dots\so{1}\xi_{i_r})\so{1}\partial_j,
     \quad 1\le r\le n,
\end{eqnarray*}
where $I=\{i_1,\dots, i_r\}\subseteq \mathcal N := \{1,\dots, n\}$,
$i_1<\dots<i_r$,
the bracketing is assumed to be right-justified.
Since the map is injective, we will identify elements of $W_n$ with their preimages
under this map.

Conformal superalgebra $K_n$ is a subalgebra of $W_n$ generated
(as a $\Bbbk[D]$-module) by the elements
\begin{equation}\label{g_I}
g_I = (2-|I|)\xi_I + (-1)^{|I|} \sum\limits_{i=1}^n (D\xi_i\xi_I\partial_i
+ \partial_i(\xi_I)\partial_i ),
\quad I\subseteq \mathcal N.
\end{equation}

Universally defined representations of $W_n$ induce finite representations
of $K_n$:
\[
\psi_k = \varphi_k\vert_{K_n} : K_n \to C_k \simeq
\begin{cases} 
 \Cend_{2^n, Q},\ Q=\diag(v,1,\dots, 1), & \mbox{if } k=1; \cr
\Cend_{\widetilde Q, 2^n},\ \widetilde Q =\diag(1,\dots, 1,v), & \mbox{if }
    k=2.\end{cases}
\]
Let $B$ stands for the set $\{g_I\mid I\subseteq \mathcal N\}$.

\begin{thm}\label{K_n thm}
\

\begin{itemize}
\item[]{\rm(1)} $(C_1,\psi_1)$, $(C_2,\psi_2)$  are associative envelopes
  of $K_n$ for $n\ne 2$. Therefore, the induced representations are irreducible.
\item[]{\rm(2)} If $n=1$ then the envelopes $(C_1,\psi_1)$ and $(C_2,\psi_2)$ are
  isomorphic. Moreover, the corresponding representation of the
  Neveu--Schwartz conformal superalgebra $K_1$ is universally defined with
  respect to $B$.
\item[]{\rm(3)} If $n>2$ then neither of $\psi_1$, $\psi_2$
is a universally defined representation with respect to $B$.
  \end{itemize}
\end{thm}

\begin{proof}
(1)
Let us compute the images of $g_I$ under $\psi_k$, $k=1,2$, as elements
of $\Bbbk[D]\otimes A_n[v]$:
\begin{eqnarray}
& \psi_1(g_I)  =
(2-|I|)(v-D)\xi_I + (-1)^{|I|} \sum\limits_{i=1}^n
\big (D\xi_i\xi_I\partial_i + \partial_i (\xi_I)\partial_i \big),
   \label{g_I-1} \\
& \psi_2(g_I)  =
(2-|I|)v\xi_I - \sum\limits_{i=1}^n \big(
D\partial_i\xi_i\xi_I + \partial_i\cdot\partial_i(\xi_I)\big ).
   \label{g_I-2}
\end{eqnarray}

For a subset $I\subseteq \mathcal N$ and an index $i\in \mathcal N$,
set
\[
\alpha(i,I) = \begin{cases}   0, & i\in I, \cr
            (-1)^{|\{j\in I\mid j<i\}|}, & i\notin I.
 \end{cases}
\]
Then $\xi_i\xi_I  = \alpha(i,I)\xi_{I\cup\{i\}}$,
$\partial_i(\xi_{I\cup\{i\}}) = \alpha(i,I)\xi_I$.
It is also easy to observe that
\[
\xi_i\partial_i \xi_J = \begin{cases}
   \xi_J, & i\in J, \\
  (-1)^{|J|}\alpha(i,J) \xi_{J\cup\{i\}}\partial_i, & i\notin J.
\end{cases}
\]

Consider the associative conformal algebra $B_n$
generated by $\psi_1(K_n)$ in $C_1$ for $n\ne 2$.
It is straightforward to compute that
\begin{eqnarray}
& \psi_1(g_\emptyset)\oo{2} \psi_1(g_{\mathcal N\setminus\{i\}})
 = (-1)^{n-i}(4-2n) \xi_{\mathcal N}\partial_i.
                                        \label{L-G} \\
& \psi_1(g_{\{i\}}) \oo{0} \xi_{J\cup\{i\}}\partial_k
  = - \alpha(i,J)\xi_J\partial_k - (-1)^{|J|+1} \xi_{J\cup \{i\}}\partial_i\partial_k.
  \label{G-G}
\end{eqnarray}
It follows from (\ref{L-G}) that
$\xi_{\mathcal N}\partial_i \in B_n$ for any $i\in \mathcal N$.

Let us show by induction on $|I|$ that
$\xi_I\partial_j$ belongs to $B_n$ for any
$I\subseteq \mathcal N$, $j\in \mathcal N$.
For $|I|=N$ we are done.
For a smaller $I$, assume that $\xi_J \partial_i \in B_n$ for all
$i\in \mathcal N$ and for all
$J\subseteq \mathcal N$ such that $|J|>|I|$. In order to show
that $\xi_I\partial_j\in B_n$ one has to consider two cases:
$j\notin I$ and $j\in I$.

If $j\notin I$ then $\xi_{I\cup\{j\}}\partial_j \in B_n$
by the induction assumption, so by (\ref{G-G})
\[
B_n\ni \psi_1(g_{\{j\}})\oo{0} \xi_{I\cup\{j\}}\partial_j
= - \alpha(j,I)\xi_I \partial_j.
\]
Hence, $\xi_I \partial_j\in B_n$.

If $j\in I$ then denote $I_j = I\setminus \{j\}$,
 $I_j^k =(I\setminus \{j\})\cup \{k\}$, and consider
\[
a(I,j):= (-1)^{|I|-1}\psi_1(g_\emptyset)\oo{2} \psi_1(g_{I_j}) =
 \psi_1(g_\emptyset) \oo{2}
  \sum\limits_{k=1}^n D(\xi_k\xi_{I_j }\partial_k).
\]
Since $a(I,j)\in B_n$ by the induction assumption, we have
\begin{eqnarray}\nonumber
a(I,j) &=& \sum\limits_{k=1}^n
 \left [
4\alpha(k,I_j) \xi_{I_j^k}\partial_k
  - 2\alpha(k,I_j)\sum\limits_{i=1}^n \xi_i\partial_i\xi_{I_j^k}\partial_k
 \right]                        \\
\nonumber &=&
 \sum\limits_{k=1}^n 2\alpha(k, I_j)
 \left[
 2\xi_{I_j^k}\partial_k - \sum\limits_{i\in I_j^k}\xi_{I_j^k}\partial_k
 -\sum\limits_{i\notin I_j^k}
   (-1)^{|I|}\xi_{I_j^k\cup \{i\}}\partial_i\partial_k
 \right]                        \\
&=&
\sum\limits_{k=1}^n 2\alpha(k, I_j)
 \left[
 (2-|I|)\xi_{I_j^k}\partial_k
 -
 \sum\limits_{i\notin I_j^k}
   (-1)^{|I|}\xi_{I_j^k\cup \{i\}}\partial_i\partial_k
 \right].
                          \label{a(i,J)}
\end{eqnarray}
Recall that $\xi_{I_j^k\cup \{i\}}\partial_k \in B_n$
for any $i\notin I_j^k$.
Now,
\[
B_n \ni g_{\{i\}}\oo{0} \xi_{I_j^k\cup \{i\}}\partial_k
= - \alpha(i,I_j^k) \xi_{I_j^k}\partial_k - (-1)^{|I|+1}
\xi_{I_{j}^k\cup \{i\}}\partial_i\partial_k.
\]
Hence,
\begin{equation}
\xi_{I_j^k\cup \{i\}} \partial_i\partial_k \equiv
(-1)^{|I|}\alpha (i, I_j^k) \xi_{I_j^k}\partial_k  \pmod{B_n}
\label{eq-comp}
\end{equation}
for any $i\notin I_j^k$.
Substitute the last relation into (\ref{a(i,J)}) and obtain
\[
a(I,j) \equiv 2\sum\limits_{k=1}^n
 \alpha (k, I_j) (2-n) \xi_{I_j^k}\partial_k   \pmod{B_n},
\]
so
\begin{equation}
\sum\limits_{k=1}^n
 \alpha (k, I_j) \xi_{I_j^k}\partial_k \in B_n.
                            \label{eq-xi1}
\end{equation}
Note that for any two different $k_1,k_2\notin I_j$
(such a pair exists since $|I|<n$) we have
$k_1\notin I_j^{k_2}$ and $k_2\notin I_j^{k_1}$.
So by (\ref{eq-comp}) we have
\begin{eqnarray}\nonumber
\alpha (k_1,I_j^{k_2})\xi_{I_j^{k_2}}\partial_{k_2}
 & \equiv&
(-1)^{|I|} \xi_{I_j\cup \{k_1,k_2\}} \partial_{k_1}\partial_{k_2}  \\
&=&
(-1)^{|I|+1} \xi_{I_j\cup \{k_1,k_2\}} \partial_{k_2}\partial_{k_1}
\equiv
-\alpha (k_2,I_j^{k_1})\xi_{I_j^{k_1}}\partial_{k_1}
                                             \pmod{B_n}.
                                             \nonumber
\end{eqnarray}
It is easy to observe that
\[
\alpha(k_1, I_j^{k_2}) = \begin{cases}
             \alpha(k_1, I_j), & k_1< k_2, \\
             -\alpha(k_1, I_j), & k_1> k_2.
\end{cases}
\]
Hence, all terms in (\ref{eq-xi1}) are equal modulo $B_n$, so
for any $k\notin I_j$ we have
\[
(n-|I|+1) \xi_{I_j^k}\partial_k \in B_n.
\]
In particular, for $k=j$ we have the required relation
$\xi_I\partial_j \in B_n$.

We have proved that all elements of the form $\xi_I\partial_j$,
$I\subseteq \mathcal N \ni j$,
belong to $B_n$. It remains to show that $(-D+v)\xi_I\in B_n$,
$I\subseteq \mathcal N$. It is enough
to consider $I=\emptyset $ and $|I|=1$.

Since $\xi_i\partial_i \in B_n$, we also have
\[
 -D+v = \frac{1}{2}\left (\psi_1(g_\emptyset) - \sum\limits_{i=1}^n
  D(\xi_i\partial_i) \right )\in B_n.
\]
Moreover,
$\psi_1(g_{\{k\}}) \equiv (-D+v)\xi_k   \pmod{B_n}$,
so
$(-D+v)\xi_k \in B_n$ for any $k\in \mathcal N$.

Therefore, the image of $K_n$ under $\psi_1$ generates the entire algebra
$C_1$. For the representation $\psi_2$ the proof is completely analogous.

(2) It was found in \citep{Ko1} that the universal envelope
$U_{N_1}(K_1,B)$ (where $N_1$ is the locality function induced by $\psi_1$)
is isomorphic to $C_1$. Hence, the associative envelope
$(C_1,\psi_1)$ corresponds to a universally defined representation.
Later we will show that this is not the case for $K_n$, $n\ge 2$.

Since $N_{C_2}(\psi_2(a),\psi_2(b))=N_1(a,b)$
for any $a,b\in B$,
the associative envelope $(C_2,\psi_2)$ has to be isomorphic to
$(C_1,\psi_1)$.

(3) It is easy to note that
for any $I\subseteq \mathcal N$
\begin{equation}\label{G-dif}
\psi_2(g_I) = \psi_1(g_I) - (n-2)D\xi_I.
\end{equation}
Hence, for $n=2$ the representations $\psi_1$ and $\psi_2$ of $K_2$
coincide, so the representation obtained is not an irreducible one.

It is not clear whether $(C_1,\psi_1)\simeq (C_2,\psi_2)$ for
$n>2$, but in any case neither of these envelopes is a universal
one. Let us compute the induced locality functions $N_k(a,b)=
N_{C_k}(\psi_k(a), \psi_k(b))$, $k=1,2$, $a,b\in B$. It is
straightforward to check that $N_1\equiv N_2\equiv N$, where
\begin{equation}
 N(g_I,g_J) = \begin{cases} 3, & I\cap J=\emptyset,\ |I\cup J|\le n-1, \cr
  2, & |I\cap J|=1 \mbox{ or } I\cap J=\emptyset,\, |I\cup J|=n, \cr
  1, & |I\cap J|=2, \cr
  0, & |I\cap J|\ge 3.
\end{cases}
\end{equation}
In particular, $N(a,b)=N(b,a)$ for all $a,b\in B$.

\begin{lem}\label{superinv}
Conformal algebra $C =\Cend_{2^n,Q}$, $Q=\diag(v,1,\dots,1)$,
has no superinvolutions if $n>1$.
\end{lem}

\begin{proof}
Consider the $\Zset_2$-grading $V=V_0\oplus V_1$ on $V=\Bbbk[D]\otimes \Bbbk^{2^n}$
that induces the $\Zset_2$-grading on $\Cend_{2^n}\simeq \Bbbk[D]\otimes A_n[v]$.
This is exactly the
canonical grading on $\wedge_n$.
Denote
\begin{eqnarray}
\nonumber
& C_{00} = \{ a\in C_0 \mid a\oo{m} V_1 = 0\, \forall m\ge 0\}, \\
& C_{01} = \{ a\in C_0 \mid a\oo{m} V_0 = 0\, \forall m\ge0\}. \nonumber
\end{eqnarray}
It is clear that $C_{00} \simeq \Cend_{2^{n-1},Q}$,
$C_{01}\simeq \Cend_{2^{n-1}}$, $C_0=C_{00}\oplus C_{01}$.
Let $\pi $ stands for the projection of $C_0$ onto $C_{01}$.

Suppose $\sigma $ is a superinvolution of $C$. Note that
$I=\pi(\sigma (C_{00}))$ is an ideal of $C_{01}$.
Hence, either $I=0$ or $I=C_{01}$.
In the first case $\sigma\vert_{C_{00}}$ is an involution of $C_{00}$.
In the last case $\sigma\pi $ is an isomorphism of $C_{00}$ and $C_{01}$.
But it was shown in \citep{BKL1} that $\Cend_{N,Q}\not\simeq \Cend_N$
if $Q$ is not invertible, and $\Cend_{N,Q}$, $Q=\diag (v,1,\dots, 1)$
has no involutions for $N>1$.
\end{proof}

It remains to apply Lemma \ref{lem lifting} to show that neither of
$\psi_k$, $k=1,2$, is a universally defined representation of $K_n$
with respect to $B$.
\end{proof}

\section*{Acknowledgements}
This work is partially supported by RFBR 05--01--00230,
Complex Integration Program SB RAS (2006--1.9) and
SB RAS grant for young researchers (Presidium SB RAS, act N.29 of
January 26, 2006).
The author gratefully acknowledges the support of the Pierre Deligne fund
based on his 2004 Balzan prize in mathematics.

The main results of this paper were presented on the Seventh Asian Symposium on
Computer Mathematics in Seoul (December 8--10, 2005). The author is very grateful to
Hyungju Park and Seok Jin Kang for their support in attending the Symposium.

\end{document}